   \documentclass[12pt]{article}    
\usepackage{latexsym, amsmath, amsthm, amsfonts}
\usepackage{enumerate, amssymb, xspace}
\usepackage{hyperref}
\usepackage{color,colordvi}
\usepackage{mathtools}   
\usepackage{tikz-cd}

\theoremstyle {definition}
\newtheorem{thm}{Theorem}[section]
\newtheorem{Lemma}[thm]{Lemma}
\newtheorem{Proposition}[thm]{Proposition}
\newtheorem{cor}[thm]{Corollary}
\newtheorem{Remark}[thm]{Remark}
\newtheorem{Definition}[thm]{Definition}

\numberwithin{equation}{section}

\newcommand\Cite[2] {\cite[#1]{#2}}

 \makeatletter
 \newcommand*{\relrelbarsep}{.386ex}
 \newcommand*{\relrelbar}{\mathrel{\mathpalette\@relrelbar\relrelbarsep}}
 \newcommand*{\@relrelbar}[2]{\raise#2\hbox to 0pt{$\m@th#1\relbar$\hss}%
     \lower#2\hbox{$\m@th#1\relbar$}}
 \providecommand*{\rightrightarrowsfill@}{%
       \arrowfill@\relrelbar\relrelbar\rightrightarrows}
 \providecommand*{\leftleftarrowsfill@}{%
        \arrowfill@\leftleftarrows\relrelbar\relrelbar}
 \providecommand*{\xrightrightarrows}[2][]{%
           \ext@arrow 0359\rightrightarrowsfill@{#1}{#2}}
 \providecommand*{\xleftleftarrows}[2][]{%
     \ext@arrow 3095\leftleftarrowsfill@{#1}{#2}}
 \makeatother

\def\act              {\triangleright}
\def\Act              {\,{\act}\,}
\def\be               {\begin{equation}}
\def\bearl            {\begin{array}{l}}
\def\bearll           {\begin{array}{ll}}
\def\botAle           {\,{\boxtimes_\Cala^{\mathrm{l.e.}}}}
\def\botAre           {\,{\boxtimes_\Cala^{\mathrm{r.e.}}}}
\def\boti             {\,{\boxtimes}\,}
\def\Boti             {{\boxtimes}}
\def\cala             {{\mathcal A}}
\def\Cala             {{\!\mathcal A}}

\def\calaopp          {{\CC{\mathcal A}}}
\def\calb             {{\mathcal B}}
\def\calbopp          {{\CC{\mathcal B}}}

\def\calc             {{\mathcal C}}

\def\calk             {{\mathcal K}}
\def\calm             {{\mathcal M}}

\def\calmopp          {{\CC{\mathcal M}}}

\def\caln             {{\mathcal N}}
\def\calnopp          {{\CC{\mathcal N}}}
\def\calp             {{\mathcal P}}
\def\calx             {{\mathcal X}}
\def\caly             {{\mathcal Y}}
\def\calz             {{\mathcal Z}}
\newcommand\cc[1]     {\overline{{#1}}}  
\newcommand\CC[1]     {\overline{{#1}}}  
\def\cent             {\mathcal{Z}}

\def\Comod            {\text{Comod}}
\def\ee               {\end{equation}}
\def\eear             {\end{array}}
\def\eq               {\,{=}\,}

\def\Fun              {{\mathcal F\!un}}
\def\Funbal           {{\mathcal F\!un}^{\rm bal}}
\def\Funle            {{\mathcal Lex}}
\def\Funlebal         {{\mathcal Lex}^{\rm bal}}
\def\Funre            {{\mathcal Rex}}
\def\Funrebal         {{\mathcal Rex}^{\rm bal}}

\def\Hom              {{\rm Hom}}

\def\id               {\mathrm{id}}
\def\Id               {\mathrm{Id}}

\def\ko               {{\ensuremath{\Bbbk}}}
\def\la               {{\rm l.a.}}
\def\Mod              {\text{Mod}}
\def\NM               {_{\caln,\calm}}
\newcommand\nxl[1]    {\\[#1mm]}
\newcommand\Nxl[1]    {\\[-1.3em]\\[#1mm]}
\def\one              {{\bf1}}
\def\opm              {^{\#}}
\def\opmm             {{}^{\#\!}}
\def\opmmm            {{}^{\#\!\!}}
\def\oti              {\,{\otimes}\,}

\def\Phile            {\Phi^{\rm l}}
\def\Phire            {\Phi^{\rm r}}
\def\Psile            {\Psi^{\rm l}}
\def\Psire            {\Psi^{\rm r}}
\def\ra               {{\rm r.a.}}
\def\ract             {\triangleleft}
\def\Ract             {\,{\ract}\,}
\def\TX               {T^{\calx}}
\def\TY               {T^{\caly}}
\def\TZ               {T}
\def\To               {\,{\to}\,}

\def\vect             {\ensuremath{\mathrm{vect}}}
\def\Vee              {{}^{\vee\!}}
\def\wPsile           {\widehat\Psi^{\rm l}}
\def\wPhile           {\widehat\Phi^{\rm l}}
\def\wPsire           {\widehat\Psi^{\rm r}}
\def\Z                {\mathbb{Z}}

\begin{document}

\thispagestyle{empty}
\begin{flushright}
   {\sf ZMP-HH/20-6}\\
   {\sf Hamburger$\;$Beitr\"age$\;$zur$\;$Mathematik$\;$Nr.$\;$828}\\[2mm] March 2020
\end{flushright}

\vskip 4.0em

\begin{center}{\bf \Large Module Eilenberg-Watts calculus }

\vskip 3em

 {\large \  \ J\"urgen Fuchs\,$^{\,a,c}, \quad$ Gregor Schaumann\,$^{\,b}, \quad$
 Christoph Schweigert\,$^{\,c}$
 }

 \vskip 12mm

 \it$^a$
 Teoretisk fysik, \ Karlstads Universitet\\
 Universitetsgatan 21, \ S\,--\,651\,88\, Karlstad
 \\[9pt] \it$^b$
 Institut f\"ur Mathematik, \ Universit\"at W\"urzburg\\ Mathematische Physik\\
 Emil-Fischer-Stra\ss e 31, \ D\,--\,97\,074 W\"urzburg
 \\[9pt] \it$^c$
 Fachbereich Mathematik, \ Universit\"at Hamburg\\ Bereich Algebra und Zahlentheorie\\
 Bundesstra\ss e 55, \ D\,--\,20\,146\, Hamburg

\end{center}

\vskip 5em

\noindent{\sc Abstract}\\[3pt]
The categorical formulation of the Eilenberg-Watts calculus relates, for any pair
of finite categories $\calm$ and $\caln$, the finite categories $\Funle(\caln,\calm)$
and $\Funre(\caln,\calm)$ of linear left or right exact functors and the Deligne 
product $\calnopp \boti \calm$ by adjoint equivalences.
We establish a variant of this calculus for the case that the finite categories
$\calm$ and $\caln$ are module categories over a finite tensor category.
This provides in particular canonical and explicitly computable equivalences
between categories of left or right exact module functors (or, more generally,
balanced functors) and certain twisted centers of bimodule categories.

\newpage

\section{Introduction}

It is textbook knowledge (see e.g.\ \Cite{Ch.\,39}{BRwi}) that any right exact linear functor
between categories of fini\-te-di\-men\-si\-o\-nal modules over fini\-te-di\-men\-si\-o\-nal
algebras over a field is isomorphic to the functor given by tensoring with a bimodule, 
while any left exact such functor is isomorphic to a Hom functor involving a bimodule. 
These isomorphisms, which date back to Eilenberg \cite{eile4} and Watts \cite{wattC}, 
are extremely useful, which partly accounts for the prominence of bimodules.
Still, from a modern perspective they are unsatisfactory -- they hide the fact that  
the underlying algebras only play the role of `coordinates'. Indeed, a Morita invariant
version exists. It is achieved by the following purely categorical formulation:
For any pair of finite linear categories $\cala$ and $\calb$ there is a commuting triangle
  \be
  \begin{tikzcd}[row sep=11ex,column sep=1.5em]
  ~ & \calaopp \boti \calb ~ \ar{dl}[xshift=-2pt]{\Phile} \ar[xshift=-2pt]{dr}[swap]{\Phire}
  & ~ \\
  \Funle(\cala,\calb) \ar[yshift=4pt]{rr}{} \ar[xshift=-12pt]{ur}[xshift=2pt]{\Psile}
  & ~ & \Funre(\cala,\calb) \ar[yshift=-4pt]{ll}{} \ar[xshift=12pt]{ul}[swap]{\Psire}
  \end{tikzcd}
  \label{eq:NoM-Lex-Rex}
  \ee
of two-sided adjoint equivalences between finite linear categories that can be 
explicitly expressed in categorical terms. Here $\Funle(\cala,\calb)$
and $\Funre(\cala,\calb)$ are categories of linear left and right exact functors, 
respectively, and $\calaopp \boti \calb$ is the Deligne product of the category $\calaopp$
opposite to $\cala$ with $\calb$. In particular, the role of the bimodule in the classical 
setting is taken over by an object of the category $\calaopp \boti \calb$.
These equivalences are given by \cite{shimi7,fScS2}
  \be
  \begin{array}{ll}
  \Phile_{} \equiv \Phile_{\!\cala,\calb}:
  & \calaopp \boti \calb \xrightarrow{~\simeq~} \Funle(\cala,\calb) \,,
  \Nxl2
  & \cc a \boti b \longmapsto \Hom_\cala(a,-) \oti b \,,
  \Nxl3
  \Psile_{} \equiv \Psile_{\!\cala,\calb}:
  & \Funle(\cala,\calb) \xrightarrow{~\simeq~} \calaopp \boti \calb \,,
  \Nxl2
  & F \longmapsto \int^{a\in\cala} \cc a \boti F(a) \,,
  \eear
  \label{Phile.Psile}
  \ee
and
  \be
  \begin{array}{lll}
  \Phire_{} \equiv \Phire_{\!\cala,\calb}:
  & \calaopp \boti \calb \xrightarrow{~\simeq~} \Funre(\cala,\calb) \,,
  \Nxl2
  & \cc a \boti b \longmapsto {\Hom_\cala(-,a)}^{\!*}_{} \oti b \,,
  \Nxl3
  \Psire_{} \equiv \Psire_{\!\cala,\calb}:
  & \Funre(\cala,\calb) \xrightarrow{~\simeq~} \calaopp \boti \calb \,,
  \Nxl2
  & G \longmapsto \int_{a\in\cala} \cc a \boti G(b) 
  \eear
  \label{Phire...Psire}
  \ee
respectively. 
	  
We refer to these equivalences as \emph{Eilenberg-Watts functors}. They can be used 
to set up an `Eilenberg-Watts calculus' which provides, at the level of linear categories,
e.g.\ a Morita invariant formulation of Nakayama functors \Cite{Sect.\,3.5}{fScS2} 
and embeds them into a Grothendieck-Verdier like picture by which they are related
to monoidal units. The Eilenberg-Watts calculus has also been a very
useful tool in the state-sum construction of a modular functor in \cite{fScS4}.
  
\medskip

In \Cite{Sect.\,4}{fScS2} the Eilenberg-Watts calculus was already applied to the situation
that the linear categories in question have the additional structure of finite module and 
bimodule categories over finite monoidal categories. It turned out that the Eilenberg-Watts 
functors can be combined with the structure of a module category to obtain a transparent
setting for several further results. Specifically, a link between the Eilenberg-Watts
calculus and Radford's theorem on the fourth power of the antipode of a Hopf algebra
was established, and an extension of the calculus to bimodule categories was discussed. This 
also led to the theory of relative Serre functors. These exist if and only if the module 
category is exact (in the sense of \Cite{Ch.\,7.5}{EGno}), and in that case they are 
related to Nakayama functors via the distinguished invertible object of the finite tensor 
category. Relative Serre functors, in turn, give rise to the notion of an inner-product 
structure on a module category over a finite tensor category \Cite{Def.\,5.2}{schaum5}, 
also known as a pivotal structure on an exact module category \Cite{Def.\,5.6}{shimi17}.
In the presence of a pivotal structure e.g.\ the inner Ends of the module category are not
just algebras, but even symmetric Frobenius algebras (see \Cite{Thm.\,3.15}{shimi20},
and \Cite{Thm.\,6.6}{schaum} for the result in the semisimple case). 

In the present contribution we further develop the interaction between module categories 
over monoidal categories and Eilenberg-Watts functors. We work over general finite tensor
categories, and in particular do not assume that they are endowed with a pivotal structure. 
As a consequence, the bidual functor and its powers constitute monoidal automorphisms of the
finite tensor category which generically are non-trivial. In the module setting, instead of 
functors and the Deligne product, the players in the game are now module functors and the 
relative Deligne product \cite{enoM} of module categories. The latter is conveniently 
described \Cite{Sect.\,2.3}{fScS} in terms of balancings, whereby also balanced functors 
and twisted centers, with the twisting by powers of the bidual functor (see Definitions 
\ref{def:enoM.3.1}\,--\,\ref{def:balancing-etc}), come into the game.

To study these structures, it is convenient to place them into a larger context involving 
(co)monads. To a comonad on a finite category one can associate induced comonads on
functor categories as well as (see Proposition \ref{proposition:induced-monad})
`transferred' comonads on other finite categories that are obtained by transporting
the comonad along two-sided adjoint equivalences, such as the Eilenberg-Watts
equivalences. Our main results are stated in Theorem \ref{thm:EW:coendzFz}: Based on  
canonical and explicitly computable equivalences between the categories of comodules
over different comonads that are obtained by such a transfer and on the resulting
universal properties (see Proposition \ref{Proposition:comonadEW}), we obtain the
\emph{module Eilenberg-Watts equivalences} \eqref{equation:factor-EW-module}, as
well as the isomorphisms \eqref{eq:Main-ModuleEW-concret} and \eqref{eq:EW:coendzz=...}
between coends over the twisted center $\cent^\kappa_{}(\calm)$ of a module category
and coends over $\calm$ itself.
As a special case, this gives a module version of the Eilenberg-Watts equivalences:
two-sided adjoint equivalences between certain twisted centers and categories of
left and right exact module functors, respectively. We present these in
Proposition \ref{prop:theprop} and Corollary \ref{cor:thecor}. We study this special
situation separately, and before the general case, because the reader might be more 
interested in module functors than in general balanced functors, and also because we 
can approach it in a more direct manner that does not require constructions involving
comonads. Along the way we also obtain a few other results that round the picture off. As
an example, we show in Proposition \ref{proposition:rel-Del} that certain twisted centers
have a natural structure of right and left exact relative Deligne products, respectively.

\medskip

In the sequel we freely use the framework of finite tensor categories and their
module and bimodule categories, as described e.g.\ in Chapters 6 and 7 of \cite{EGno},
as well as properties of ends and coends. The latter can e.g.\ be found in 
\cite{fuSc23} and in Section 2.2 of \cite{fScS2}. As one crucial relation let us 
mention explicitly that, as a version of the (co-)Yoneda lemma, for $F$ a linear
functor between finite categories $\cala$ and $\calb$ there are natural 
isomorphisms \Cite{Cor.\,1.4.5\,\&\,Ex.\,1.4.6}{RIeh}
  \be
  \int^{a \in \cala}\! \Hom_\cala(a,-) \otimes F(a) \,\cong\, F \,\cong
  \int_{\!a \in \cala}\! \Hom_\cala(-,a)^*_{} \otimes F(a)
  \label{eq:Yoneda}
  \ee
of linear functors. We also assume that the reader is familiar with the Deligne product 
$\boxtimes$ of finite linear categories \Cite{Sect.\,5}{deli}, which is universal for 
left exact as well as for right exact functors, with the relative Deligne product 
(see e.g.\ \Cite{Sect.\,2.5}{fScS}) $\boxtimes_\cala$ of right and left modules 
over a finite tensor category $\cala$, and with basic aspects of monads and modules
over monads (see e.g.\ \cite{brVi5}).

Throughout, we fix an algebraically closed field \ko\ and denote by $V^*_{}$ the 
dual of a \ko-vector space $V$. All vector spaces are fini\-te-di\-men\-si\-o\-nal 
\ko-vector spaces; by $\vect$ we denote the symmetric monoidal category of such vector
spaces and linear maps. All categories are taken to be abelian and finite \ko-linear,
and all functors and natural transformations are taken to be \ko-linear. 
The linear category opposite to $\cala$ is denoted by $\calaopp$. Our
conventions for dualities are as follows: The left and right dual of an object
$x$ are denoted by $\Vee x$ and $x^\vee$, respectively; the corresponding
evaluation and coevaluation morphisms lie in the morphism spaces
$\Hom(x \oti \Vee x,\one)$ and $\Hom(\one,\Vee x \oti x)$ for the left duality,
and in $\Hom(x^\vee \oti x,\one)$ and $\Hom(\one,x \oti x^\vee)$ for the right duality.


\section{More about the Eilenberg-Watts calculus}

Before turning to the context of module categories, we make a few further 
observations that are easy consequences of the treatment in \cite{fScS2}.

Let us first point out that $\Psile_{\!\cala,\calb}$ 
as defined in \eqref{Phile.Psile} maps a left exact functor
$F$ to a coend $\int^{a\in\cala} \cc a \boti F(a)$, that is, to an object $C$ in
$\calaopp \boti \calb$ endowed with dinatural structure morphisms
$\cc a\,{\boxtimes} F(a) \,{\to}\, C$ for every $a\,{\in}\,\cala$. Since $\Psile_{}$ 
is an equivalence of categories, it follows conversely that every left exact 
functor $F\colon \cala\To\calb$ between finite
categories must come with such canonical structural morphisms (that is, morphisms
$\Phile_{\!\cala,\calb}(\cc a \boti F(a)) \eq \Hom_\cala(a,-) \oti F(a)
\,{\xrightarrow{~~}}\, F$ for $a \,{\boxtimes} F(a) \,{\in}\, \calaopp \boti \calb$).
This is indeed the case, even for any linear functor, and one recovers the Yoneda
isomorphisms \eqref{eq:Yoneda}. Analogously, the equivalence $\Psire$ endows every
right exact functor between finite categories with the canonical structure of an end, 
given by the second Yoneda isomorphism in \eqref{eq:Yoneda}.

Next we note the following parameter version of the Eilenberg-Watts functors, from which
the equivalences \eqref{Phile.Psile} are recovered by taking $\calp$ to be $\vect$:

\begin{Lemma}\Cite{Lemma\,B.4}{fScS4} \label{Lemma:EW-adjun}
Let $\cala$, $\calb$ and $\calp$ be finite categories. There are adjoint equivalences 
  \be
  \begin{array}{cl}
  \Funle(\calp \boti \cala, \calb) \,\simeq\, \Funle(\calp, \calaopp \boti \calb) ~
  & \quad \text{and}
  \Nxl3
  \Funre(\calp \boti \cala, \calb) \,\simeq\, \Funre(\calp, \calaopp \boti \calb) \,.
  \eear
  \label{eq:EW-inLEXREX}
  \ee
\end{Lemma}

\begin{proof}
This follows immediately by observing that
  \be
  \Funle(\calp \boti \cala, \calb) \,\simeq\, \cc{\calp \boti \cala} \boxtimes \calb
  \,\simeq\, \cc{\calp} \boxtimes (\cc{\cala} \boti \calb)
  \,\simeq\, \Funle(\calp, \calaopp \boti \calb) \,,
  \ee
where the first and third equivalences are the ordinary Eilenberg-Watts functors
\eqref{Phile.Psile},
while the one in the middle comes from the associativity property of the Deligne product. 
\end{proof}

One virtue of the Eilenberg-Watts calculus is that it allows one to manipulate ends and
coends in ways that are not covered by the Yoneda isomorphisms \eqref{eq:Yoneda}.
Specifically, according to Proposition 3.4 of \cite{fScS2} 
for any left exact functor $F \,{\in}\, \Funle(\cala,\calb)$, one has
  \be
  \Hom_{\calaopp\boxtimes\calb} \big(-\,, \mbox{\large$\int$}^{a\in\cala}_{}
  \cc a \,{\boxtimes} F(a) \big) \,\cong
  \int^{a\in\cala}\!\! \Hom_{\calaopp\boxtimes\calb} \big(-\,, \cc a \,{\boxtimes} F(a) \big) \,,
  \label{eq:fScS2Prop3.4}
  \ee
and likewise a similar isomorphism for right exact functors.
As a consequence, while left exact functors always commute with ends and right exact functors 
commute with coends, in our situation we have in addition

\begin{Lemma} \label{Lemma:coend-commut}
Let $\cala,\calb,\calk$ and $\calx$ be finite linear categories.
\\[2pt]
{\rm (i)}
For any left exact functor $F \,{\in}\, \Funle(\cala \boti \calaopp\boti \calk, \calx)$ we have
  \be
  \int^{a \in \cala}\!\!\! F(a \boti \cc{a} \boti k)
  \,\cong\, F\big( \mbox{\large$\int$}^{a\in\cala}_{} a \boti \cc{a} \boti k \big)    
  \ee
for all $k \,{\in}\, \calk$. Analogously, for any right exact functor
$G \,{\in}\, \Funre(\cala \boti \calaopp\boti \calk, \calx)$ there is an isomorphism
$\int_{a} G(a \boti \cc a \boti k) \,{\cong}\, G(\int_{a} a \boti \cc a \boti k)$.
\\[2pt]
{\rm (ii)}
Let $F\colon \cala \boti \calaopp \boti \calb \boti \calbopp \boti \calk \,{\to}\, \calx$ 
be a left exact or a right exact functor. Then there is a canonical natural isomorphism
  \be
  \int^{a \in \cala}\!\!\!\! \int_{\!b \in \calb}\! F(a \boti \cc a \boti b \boti \cc b \boti {-})
  \,\cong \int_{\!b \in \calb}\! \int^{a \in \cala}\!\!\!
  F(a \boti \cc a \boti b \boti \cc b \boti {-})
  \ee
of functors from $\calk$ to $\calx$.
\end{Lemma}

\begin{proof}
(i)\, By combining the Yoneda isomorphism \eqref{eq:Yoneda}, the Fubini theorem for coends 
and the isomorphism \eqref{eq:fScS2Prop3.4} we obtain 
  \be
  \begin{array}{rl} \displaystyle
  \int^{a \in \cala}\!\!\! F(a \boti \cc{a} \boti k) \!\! & \displaystyle
  \cong \int^{a\in\cala}\!\!\!\! \int^{y \in \cala\boxtimes\calaopp\boxtimes\calk}\!\!
  \Hom(y,a \boti \cc{a} \boti k) \otimes F(y)
  \Nxl1 & \displaystyle
  \cong \int^{y \in \cala\boxtimes\CC{\cala}\boxtimes\calk}\!\!\!\! \int^{a\in\cala}\!\!
  \Hom(y,a \boti \cc{a} \boti k) \otimes F(y)
  \Nxl1 & 
  \cong {\displaystyle\int^{y \in \cala\boxtimes\calaopp\boxtimes\calk}}\!\!
  \Hom(y,\int^{a\in\cala}\! a \boti \cc{a} \boti k) \otimes F(y)
  \Nxl2 & 
  \cong\, F(\int^{a\in\cala}\! a \boti \cc{a} \boti k) \,. 
  \eear
  \ee
The  case of a right exact functor is treated dually. 
 \\[2pt]
(ii)\, If $F$ is left exact, then it commutes with ends in the following sense
\Cite{Lemma\,2.6}{fScS2}: for any
$H\colon \calaopp\,{\times}\,\cala \To \cala \boti \calaopp \boti \calb \boti \calbopp \boti \calk$
such that the end $\int_{a\in\cala}\! H(\cc a,a)$ exists, there is a canonical isomorphism
$F \big(\int_{a\in\cala}\! H(\cc a,a) \big) \,{\cong} \int_{a\in\cala} F(H(\cc a,a))$.
Hence by invoking (i) we obtain 
  \be
  \begin{array}{rl} \displaystyle
  \int^{a\in\cala}\!\!\!\! \int_{\!b\in\calb}\! F(a \boti \cc{a} \boti b \boti \cc{b} \boti k) \!\!&
  \cong\, F\big( (\int^{a\in\cala}\! a \boti \cc{a}) \,{\boxtimes} \int_{b\in\calb}(b \boti \cc{b})
  \boti k\big)
  \nxl1 & \displaystyle
  \cong \int_{\!b\in\calb}\! \int^{a\in\cala}\!\!\! F(a \boti \cc{a} \boti b \boti \cc{b} \boti k) 
  \eear
  \ee
for any $k \,{\in}\, \calk$.
Similarly, if the functor $F$ is right exact, it commutes with coends, and in the 
present situation, by (i), with ends as well; thus  the argument works analogously.
\end{proof}

Since the triangle \eqref{eq:NoM-Lex-Rex} relates left and right exact
functors, and since the left adjoint of a left exact functor is right exact and vice versa,
the Eilenberg-Watts functors are particularly useful in combination with adjoints of functors.
For instance, as seen in Lemma 3.8 of \cite{fScS2}, for $\cala$ and $\calb$ finite 
linear categories and $F \,{\in}\, \Funle(\cala,\calb)$ we have an isomorphism
  \be
  \Psile(F) \equiv \int^{a\in\cala}\!\! \cc a \boti F(a)
  \,\cong\, \int^{b\in\calb}\! \cc{F^\la(b)} \boti b
  \label{eq:lem:fScS2:3.8}
  \ee
of coends, as well as a similar isomorphism of ends for right exact functors.
Below we will need the following modest generalization of this result:

\begin{Lemma} \label{lem:3.8plus}
Let $\cala$ and $\calb$ be finite linear categories. For any two left exact functors
$F_1\colon \cala\,{\to}\,\calb$ and $F_2\colon \calb\,{\to}\,\cala$ there is an isomorphism
  \be
  \int^{b\in\calb}\! \cc{F_1^{\la}(b)} \boti F_2(b)
  \,\cong \int^{a\in\cala}\!\! \cc{a} \boti F_2{\circ}F_1(a)
  \ee
of coends. In the case of right exact functors there is a similar isomorphism of ends.
\end{Lemma}


\section{Balanced functors and twistings} \label{sec:BalF}

We are now interested in the situation that the finite categories involved in the 
Eilenberg-Watts equivalences of Lemma \ref{Lemma:EW-adjun} come equipped  with actions 
of finite tensor categories. The resulting module versions of the Eilenberg-Watts 
equivalences involve module functors, balanced functors and (twisted) centers. In the 
present section we recall these and related pertinent notions. We formulate the categories 
that will appear in the equivalences in the convenient language of monads and comonads.
          
Let thus now $\cala$ and $\calb$ be finite tensor categories, and let
$\calm\,{\equiv}\,(\calm,\act,\ract)$ be an $\cala$-$\calb$-bi\-module category. We do not
require the monoidal categories $\cala$ and $\calb$ to be endowed with a pivotal structure. 
As a consequence, by acting with powers of the monoidal functor of taking biduals we get 
generically different bimodules structures from a given one. To handle these different 
bimodules, we introduce the notation
  \be
  a^{[\kappa]} := a^{\vee...\vee}
  \ee
for the $\kappa$-fold right dual of $a\,{\in}\,\cala$,
and ${}^{[\kappa]}a \,{:=}\, {}^{\vee...\vee\!}a$ for the $\kappa$-fold left dual.
(Via the identification $a^{[-\kappa]} \,{=}\, {}^{[\kappa]}a$ 
these notations make sense also for negative values of $\kappa$.)
Then for any pair $(\kappa,\kappa') \,{\in}\, 2\Z \,{\times}\, 2\Z$ we denote by
${}_{}^{\kappa\!\!}\calm{}^{\kappa'}_{}$ the $\cala$-$\calb$-bimodule bimodule for which the
left and right actions on $\calm$ are twisted by the $\kappa$-fold left and $\kappa'$-fold
right dual, respectively, that is, $\cala$ acts as $m \,{\mapsto}\, {}^{[\kappa]}a \Act m$
and $\calb$ as $m \,{\mapsto}\, m \Ract b^{[\kappa']}$.
Thus in particular ${}_{}^{0\!\!}\calm{}^{0} \eq \calm$;
we also abbreviate ${}_{}^{0\!\!}\calm{}^{\kappa} \,{\equiv}\, \calm{}^{\kappa}$
and ${}_{}^{\kappa\!\!}\calm{}^{0} \,{\equiv}\, {}_{}^{\kappa\!\!}\calm$.
In the same vein we endow, for any pair $(\kappa,\kappa') \,{\in}\, 2\Z{+}1 \,{\times}\,
2\Z{+}1$, the opposite category $\calmopp$ with the structure of a $\calb$-$\cala$-bimodule
by defining the actions of $a\,{\in}\,\cala$ and $b\,{\in}\,\calb$ to be
  \be
  \cc{m} \Ract a := \cc{ a^{[\kappa']} \Act m}
  \qquad\text{and}\qquad  b \Act \cc m := \cc{m\Ract {}^{[\kappa]}b} \,,
  \ee
respectively. We denote the so obtained bimodule categories by
${}_{}^{\kappa\!}\CC\calm{}^{\kappa'}_{}$. By omitting one of the actions,
analogous definitions apply to left and right module categories. In the common situation
that we want to pass from a left module $\caln$ to a right module with underlying 
category being the opposite $\cc\caln$, we also use the special notation
  \be
  {}_{}^{}\CC\caln{}^{1}_{} =: \caln\opm \qquad \text{and} \qquad 
  {}_{}^{}\CC\caln{}^{-1}_{} =: \opmmm\caln \,.
  \label{eq:opm-opmmm}
  \ee
          
In the context of (bi)module categories, it is natural to consider two types of extra
structure on a functor. The first type is given by the notion of a module functor:
For $\calm$ and $\caln$ left $\cala$-modules,
the category of \emph{module functors} from $\caln$ to $\calm$, denoted by
$\Fun_{\cala}(\caln,\calm)$, has as objects functors $F\colon \caln \To \calm$ that come 
with a family of natural  isomorphisms $F(a \Act n) \,{\cong}\, a \Act F(n)$
for $a \,{\in}\, \cala$ and $n \,{\in}\, \caln$ that is coherent with respect to 
the tensor product of $\cala$ \Cite{Def.\,7.2.1}{EGno}. 
In our context we are particularly interested in the finite subcategories of
left and right exact module functors, which we denote by $\Funle_{\cala}(\caln,\calm)$
and $\Funre_{\cala}(\caln,\calm)$, respectively.

The second type of functors with extra structure are balanced functors:

\begin{Definition}\label{def:enoM.3.1} \Cite{Def.\,3.1}{enoM} 
\\[2pt]
(i) For a right module $\caln \,{=}\, (\caln,\ract)$ 
and a left module $\calm \,{=}\, (\calm,\act)$ over a finite tensor category $\cala$, 
a \emph{balanced functor} from $\caln \boti \calm$ to a finite category $\calx$ is a linear 
functor $F\colon \caln \boti \calm \To \calx$ together with a natural family of isomorphisms
  \be
  F(n\Ract a \boti m) \xrightarrow{~\cong~} F(n \boti a \Act m)
  \ee
for $a\,{\in}\,\cala$, $m\,{\in}\,\calm$ and $n\,{\in}\,\caln$. 
\\[2pt]
(ii) The category $\Funbal(\caln \boti \calm, \calx)$ of balanced functors has
as objects the balanced functors from $\caln \boti \calm$ to $\calx$ and as morphisms
those natural transformations between balanced functors which are compatible with the
balancings.
\end{Definition}

Again there are left and right exact versions. A natural generalization is to consider
balanced functors from an arbitrary $\cala$-bimodule category $\calm$ to $\calx$. Still more 
generally, we set:

\begin{Definition} \label{def:twistedbalanced}
Let $\calm$ be a bimodule over a finite tensor category $\cala$, and $\calx$ a
finite category. For $\kappa \,{\in}\, 2\Z$, the category of \emph{$\kappa$-twisted 
balanced functors} from $\calm$ to $\calx$, denoted $\Fun^{\kappa}(\calm, \calx)$, 
has as objects functors $F$ that come with coherent isomorphisms 
  \be
  F(a\Act m) \xrightarrow{~\cong~} F(m \Ract a^{[\kappa]})
  \label{eq:twbalanced}
  \ee
for $a \,{\in}\, \cala$.
\end{Definition}

The module Eilenberg-Watts calculus to be developed below relates module functors
and balanced functors with (twisted) centers. We first recall
          
\begin{Definition}\label{def:balancing-etc} \Cite{Def.\,3.2,3.4}{fScS4}
Let $\cala$ be a finite tensor category and $\calm$ a finite $\cala$-bi\-module,
and let $m\,{\in}\,\calm$. 
\\[2pt]
{\rm (i)}
A \emph{balancing} on $m$ is a natural family
$(\sigma \eq (\sigma_a\colon a{\act}m\To m{\ract}a)_{a\in \cala})$ of isomorphisms 
in $\calm$ that is coherent with respect to the tensor product of $\cala$, i.e.\ 
(taking the action of the monoidal unit to be strict) satisfies $\sigma_1 \eq \id_m$ and
  \be
  \sigma_{a\otimes a'} = (\sigma_a \Ract a') \circ (a \Act \sigma_{a'}) : \quad
  (a{\otimes}a') \Act m \to m \Ract (a{\otimes}a')
  \ee
for all $a,a' \,{\in}\, \cala$ (here the constraint morphisms of the bimodule $\calm$ 
are suppressed).
\\[2pt]
{\rm (ii)}
For $\kappa \,{\in}\, 2\Z$, a $\kappa$-\emph{twisted balancing} on $m$ is a natural family
of isomorphisms 
\be
\label{eq:kappa-twist-bal}
  {}^{[\kappa-2]}a \Act m \xrightarrow{~\cong~} m \Ract a
  \ee
in $\calm$, for $a\,{\in}\,\cala$, that is coherent with respect to the tensor product of $\cala$.
\\[2pt]
{\rm (iii)}
The $\kappa$-\emph{twisted center} of $\calm$ is the category of objects of $\calm$
endowed with $\kappa$-twisted balancings. We denote this category by
$\cent_{}^\kappa(\calm) \,{\equiv}\, \cent_\Cala^\kappa(\calm)$
(in \cite{fScS4} this category is denoted by $\calm \,{\stackrel\kappa\boxtimes}\,$).
\end{Definition}

The convention of having a shift by $-2$ in Equation \eqref{eq:kappa-twist-bal} is chosen
in such a way that the $0$-twisted center is universal for balanced left exact functors,
see Proposition \ref{proposition:rel-Del} below, and that $\kappa$-balanced left exact
functors correspond to functors from the $\kappa$-twisted center, as will be seen in Theorem
\ref{thm:EW:coendzFz}. That our conventions favor left exact over right exact functors can
be traced back to the more common use of $\Hom$-spaces as compared to dual $\Hom$-spaces. Also
note that it is sufficient to define balancings just as a family of morphisms; owing to the 
dualities of $\cala$ these morphisms are automatically invertible \Cite{Lemma\,3.3}{fScS4}.
Balanced functors and twisted centers are related via the relative Deligne product
\cite{enoM}, which comes in two variants.

\begin{Definition} \label{definition:relDel}
Let $\calm$ and $\caln$ be a right and left $\cala$-module, respectively. The 
\emph{right exact relative Deligne product $\calm \botAre \caln $} is the finite category
characterized by the following universal property: It is equipped with a balanced 
right exact functor $B_{\calm,\caln}\colon \calm \boti \caln \To \calm \botAre \caln$,
such that for every finite category $\calx$, pre-composition with $B_{\calm,\caln}$ 
furnishes an adjoint equivalence
  \be
  \Funrebal(\calm \boti \caln, \calx) \simeq \Funre(\calm \botAre \caln, \calx) \,.       
  \ee
Analogously, the \emph{left exact relative Deligne product} $\calm \botAle \caln$ 
is universal for left exact balanced functors. 
\end{Definition}
 
In \Cite{Prop.\,3.8}{enoM} the right exact relative Deligne product is shown
to be related to the twisted center.\,%
 \footnote{~The twist is omitted in \cite{enoM}, see \cite{fScS}.}

\begin{Proposition} \label{proposition:rel-Del}
The $0$-twisted center $\cent_\Cala^0(\calm \Boti \caln)$ has a natural structure 
of a left exact relative Deligne product of $\calm$ and $\caln$. The $4$-twisted center
$\cent_\Cala^4(\calm \Boti \caln)$ has a natural structure of a right exact relative Deligne
product of $\calm$ and $\caln$. The universal balanced functors are provided by the left 
adjoint, respectively right adjoint, of the forgetful functor from $\cent_\Cala^0(\calm 
\Boti \caln)$ and from $\cent_\Cala^4(\calm \Boti \caln)$ to $\calm \Boti \caln$, respectively.
\end{Proposition}

\begin{proof}
We only recall the main idea of the proof. Let $F\colon \calm \boti \caln \To \calx$ 
be balanced right exact. Then it admits a right adjoint $F^{\ra}$, and
for all $x \,{\in}\, \calx$, the object $F^{\ra}(x)$ in $\calm \boti \caln$
naturally has the structure of an object in $\cent_\Cala^4(\calm \boti \caln)$. Indeed, 
for all $m \boti n \,{\in}\, \calm \boti \caln$ there are natural isomorphisms
  \be
  \begin{array}{ll}
  \Hom(m \boti n, a \act F^{\ra}(x)) \!\!& \cong \Hom(m \boti a^{\vee} \Act n, F^\ra(x)) 
  \Nxl3 &
  \cong \Hom(F(m \boti a^{\vee} \Act n), x)
  \Nxl1 &
  \!\! \stackrel{\eqref{eq:twbalanced}}\cong \! \Hom(F(m \Ract a^{\vee} \boti n), x)
  \Nxl3 &
  \cong \Hom( m \Ract a^{\vee} \boti n, F^\ra(x))
  \Nxl3 &
  \cong \Hom(m \boti n, F^{\ra}(x) \ract a^{\vee \vee}) \,.
  \end{array}
  \ee
These isomorphisms endow $F^\ra(x)$ with the structure of an object in
$\cent_\Cala^4(\calm \Boti \caln)$ and give one part of the adjoint equivalence 
from Definition \ref{definition:relDel}.
\\
Analogously, if $F$ is left exact and balanced, we obtain objects
$F^{\la}(x) \,{\in}\, \cent_\Cala^0(\calm \Boti \caln)$. The universal property 
follows in the same way as in the proof of \Cite{Prop.{} 3.8}{enoM},
and the characterization of the universal balanced functors as in \Cite{Prop.{} 2.18}{fScS}.
\end{proof}

\begin{Remark}\label{Remark:shift-4}
For any $\cala$-bimodule $\calm$ there is an equivalence 
  \be
  F_\calm :\quad \cent_\Cala^0(\calm) \xrightarrow{\,\simeq}\, \cent_\Cala^4(\calm)
  \ee
of linear categories which makes use of the distinguished invertible object 
$D \,{\in}\, \cala$. Recall from \cite{etno2} that the object $D$ comes with coherent 
natural isomorphisms $D \oti {}^{\vee\vee\!}a \,{\cong}\, a^{\vee \vee} \oti D$ for 
all $a \,{\in}\, \cala$. It follows that by setting $F_\calm(m) \,{:=} D \act m$ we 
get an equivalence between the two twisted centers as finite categories. 
\end{Remark}

We now show that the categories of module functors, balanced functors, and centers
can all be expressed as (co)modules over suitable (co)monads. Recall \Cite{Ch.\,9}{TUvi}
that the \emph{central comonad} of a finite tensor category $\cala$ is the endofunctor
$Z\colon \cala\To\cala$ that maps objects as
  \be
  Z: \quad b ~\longmapsto \int_{\!a\in \cala}\! a\, {\otimes}\, b \,{\otimes}\, a^\vee ,
  \label{eq:defZ}
  \ee
and that the category of comodules over this comonad is equivalent to the center 
$\calz(\cala)$ of $\cala$.
We consider variants of the central comonad that are twisted by powers of the bidual.
Let $\calm$ be a bimodule category over finite tensor category $\cala$. There is then
a comonad analogous to \eqref{eq:defZ} in which the tensor products in the end 
\eqref{eq:defZ} are replaced by the left and right actions of $\cala$ on $\calm$; for any
$\kappa \,{\in}\, 2\Z$ this can be generalized \cite{fScS4} to a comonad 
  \be
  Z_{[\kappa]}:\quad  m ~\longmapsto \int_{\!a\in\cala}\! a \act m \ract a^{[\kappa-1]}
  \label{eq:def:comonadZn}
  \ee
on $\calm$. 

Let now $\calm$ be an $\cala$-bimodule, and $\calk$ and $\caln$ left $\cala$ modules.
Then the category $\Fun(\calk,\caln)$ becomes an $\cala$-bimodule by setting
  \be
  a_{1} \act F \ract a_{2}(k) :=\, a_{1} \act F(a_{2} \act k) 
  \label{eq:comonad-Fun}
  \ee
for $a_1,a_2 \,{\in}\, \cala$. Likewise, for any finite category $\calx$, $\Fun(\calm, \calx)$
is an $\cala$-bimodule, with 
  \be
  (a_{1} \act G \ract a_{2})(m) := G(a_{2} \act m \ract a_{1}) \,.
  \label{eq:comonad-Fun-bi}
  \ee
We then obtain the following categories of comodules over twisted central comonads on these 
$\cala$-bimodules:

\begin{Lemma}\label{lem:comodules} Let $\kappa \,{\in}\, 2\Z$.
\\[3pt]
(i) The category of comodules over $Z_{[\kappa]}$ on an $\cala$-bimodule $\calm$ 
is equivalent to the $\kappa$-twisted center $\calz^{\kappa}(\calm)$.
\\[3pt]
(ii) For $\calk$ and $\caln$ left $\cala$-modules, the category of comodules over 
$Z_{[\kappa]}$ on the functor category $\Fun(\calk,\caln)$ is equivalent to the
category $\Fun_{\cala}({}_{\cala}^{}\calk, {}_{}^{\kappa-2}\!\!\!\!{}_{\cala}^{}\caln)$   
of module functors.
\\[3pt]
(iii) For $\calm$ an $\cala$-bimodule, the category of comodules over $Z_{[\kappa]}$ on 
$\Fun(\calm,\calx)$ is the category of twisted balanced functors $\Fun^{2-\kappa}(\calm,\calx)$.
\end{Lemma}

\begin{proof}
The structure of a comodule over $Z_{[\kappa]}$ on $m$ is equivalent to a 
coherent family of morphisms $m \To a \Act m \Ract a^{[\kappa-1]}$. By adjunction
this is, in turn, equivalent to a coherent family ${}^{[\kappa-2]}a \Act m \To m \Ract a$.
This shows claim (i) (see also \Cite{Cor.\,B.3}{fScS4}).
 \\[2pt]
As for claim (ii), notice that according to the formulas \eqref{eq:def:comonadZn} and 
\eqref{eq:comonad-Fun} the structure of a $Z_{[\kappa]}$-comodule on a functor 
$F \,{\in}\, \Fun(\calk,\caln)$ gives rise to a coherent 
family $F(k) \To a \Act F(a^{[\kappa-1]}\Act k)$. This gives coherent isomorphisms 
$F(a \Act k) \,{\cong}\, {}^{[\kappa-2]}a \act F(k)$, as required for a module functor
from ${}_{\cala}^{}\calk$ to ${}_{}^{\kappa-2}\!\!\!\!{}_{\cala}^{}\caln$.
\\[2pt]
Concerning claim (iii), we note that by \eqref{eq:def:comonadZn}
and \eqref{eq:comonad-Fun-bi} a functor $F \,{\in}\, \Fun(\calm,\calx)$ with comodule 
structure over the comonad $Z_{[\kappa]}$ comes with coherent morphisms 
$F(m) \To F(a^{[\kappa-1]}\Act m \Ract a)$. Using the
evaluation morphism, this provides us with a sequence of morphisms $F(m \Ract a^{\vee}) 
	\,{\to} $\linebreak[0]$
F({a}^{[\kappa-1]} \Act (m \Ract a^{\vee}) \Ract a)
\To F({a}^{[\kappa-1]} \Act m )$. The composite of these morphisms is an isomorphism 
and a balancing, as required. The converse follows analogously.
\end{proof}


\section{Eilenberg-Watts calculus for module functors}\label{sec:EWm}

To proceed we invoke Corollary 4.3 of \cite{fScS2}, which states that for any
$\cala$-$\calb$-bimodule category $\calm \,{=}\, (\calm,\act,\ract)$ over finite tensor
categories there are, coherently with respect to the monoidal structures of $\cala$ 
and $\calb$, natural isomorphisms
  \be
  \bearl \displaystyle
  \int_{m\in\calm}\! \cc m \boxtimes a \Act m \Ract b
  \,\cong \int_{m\in\calm}\!\! \cc{{}^{\vee\!}a \Act m \Ract b^\vee} \boxtimes m
  \qquad\text{and}
  \Nxl3 \displaystyle
  \int^{m\in\calm}\!\! \cc m \boxtimes a \Act m \Ract b
  \,\cong \int^{m\in\calm} \cc{{a^{\vee\!} \Act m \Ract {}^{\vee\!}b}} \boxtimes m
  \label{eq:Action-boxtimes}
  \eear
  \ee
(be aware of the different appearance of left and right duals in the two formulas).
Let now $F$ and $G$ be left and right exact module functors, respectively, between
finite left $\cala$-modules $\caln \,{\equiv}\, {}_\cala\caln$ and 
$\calm \,{\equiv}\, {}_\cala\calm$, with $\cala$ a finite tensor category. We use
the notations $\caln\opm$ and $\opmmm\caln$ introduced in \eqref{eq:opm-opmmm}.
By applying the result \eqref{eq:Action-boxtimes} to the $\cala$-bimodules 
$\opmmm\caln \boti \calm$ and $\caln\opm \boti \calm$, respectively, we get isomorphisms
  \be
  \bearll
  \Psile\NM(F(a\Act{-})) \!\!\!& \displaystyle
  = \int^{n\in\caln}\!\! \cc n \,\boti F(a\Act n)
  \cong \!\int^{n\in\caln}\!\! \cc n \boxtimes a\,{\act} F(n)
  \Nxl1 & \displaystyle
  \!\! \stackrel{\eqref{eq:Action-boxtimes}} \cong \!
  \int^{n\in\caln}\!\! \cc{a^\vee\Act n} \,\boti F(n)
  \,\cong \int^{n\in\caln}\!\!\! n\opm {\Ract}\, a \boti F(n)
  \eear
  \ee
and
  \be
  \bearll
  \Psire\NM(G(a\Act{-})) \!\!\!& \displaystyle
  = \int_{\!n\in\caln}\! \cc n \,\boti G(a\Act n)
  \cong \!\int_{\!n\in\caln}\! \cc n \boxtimes a\,{\act} G(n)
  \Nxl2 & \displaystyle
  \!\! \stackrel{\eqref{eq:Action-boxtimes}} \cong \!
  \int_{\!n\in\caln}\! \cc{\Vee a\Act n} \boxtimes G(n)
  \,\cong \int_{\!n\in\caln}\! \opmm n\Ract a \boxtimes G(n) \,,
  \eear
  \ee
respectively, coherent in $a \,{\in}\, \cala$.
Together with the exactness of the action functor, which allows one to exchange it 
with (co)ends, this implies that the object $\Psile(F) \,{\in}\, \calnopp\boti\calm$ has a 
natural structure of an object in the center $\calz(\caln\opm \boti \calm)$, while
$\Psire(G)$ has a natural structure of an object in $\calz(\opmmm\caln \boti \calm)$.

Conversely, given an object $n\opm \boti m$ in the center 
$\calz(\caln\opm \boti \calm)$, there are coherent balancing isomorphisms
$\cc{a^\vee{\act} n} \boti m \,{\cong}\,\cc n \boti a\Act m$, and hence we have
  \be
  \bearll
  \Hom_\caln(n, a \Act {-}) \oti m \!\!\!& \cong \Hom_\caln(a^\vee{\act} n,-) \oti m
  \Nxl3 &
  = \Phile\NM(\cc{a^\vee{\act} n} \boti m) \cong\, \Phile\NM(\cc n \boti a\Act m) 
  \Nxl2 &
  \phantom{= \Phile\NM(\cc{a^\vee{\act} n} \boti m)}
  = \Hom_\caln(n,-) \oti\, a \Act m\,
  \eear
  \ee
for any $a\,{\in}\,\cala$. Thus the left exact functor $\Hom_\caln(n,-) \oti m$ is 
endowed with the structure of a module functor from ${}_\cala\caln$ to ${}_\cala\calm$.
Analogously one obtains, invoking the definition of $\Phire\NM$, that
  \be
  \Hom_\caln(a \Act {-},n)^* \oti m \,\cong\, \Hom_\caln(-,n)^* \oti \,a \Act m
  \ee
for $\opmm n\boti m \,{\in}\, \calz(\opmmm\caln\boti\calm)$, showing that the right 
exact functor $\Hom_\caln(-,n)^* \oti m$ comes with the structure of
a module functor from ${}_\cala\caln$ to ${}_\cala\calm$, too.

Taken together, we have the following module version of the Eilenberg-Watts equivalences:

\begin{Proposition}\label{prop:theprop}
For finite left modules $\caln$ and $\calm$ over a finite tensor category $\cala$
the equivalences \eqref{Phile.Psile} induce two-sided adjoint equivalences 
  \be
  \begin{tikzcd}[column sep=-2.8ex,row sep=0.8em]
  ~ & \calz(\caln\opm \boti \calm) ~~~~~~ \ar{ddl}[xshift=-2pt]{\Phile} & ~
  & ~~~~~~ \calz(\opmmm\caln \boti \calm) \ar[xshift=-2pt]{ddr}[swap]{\Phire} & ~
  \\
  ~ & ~ & \mbox{$\quad$and$\quad$} & ~ & 
  \\
  \Funle_\Cala(\caln,\calm) \ar[xshift=-12pt]{uur}[xshift=2pt]{\Psile} & ~ & ~ & ~
  & \Funre_\Cala(\caln,\calm) \ar[xshift=12pt]{uul}[swap]{\Psire}
  \end{tikzcd}
  \label{LexZ--RexZ}
  \ee
of finite categories.
\end{Proposition}

In case the finite tensor category $\cala$ is pivotal, the pivotal structure 
gives rise to a distinguished equivalence between the right $\cala$-modules 
$\caln\opm$ and $\opmmm\caln$, and hence also between the finite categories 
$\calz(\caln\opm \boti \calm)$ and $\calz(\opmmm\caln \boti \calm)$. Thus we have

\begin{cor}\label{cor:thecor}
For finite left modules $\caln$ and $\calm$ over a \emph{pivotal} finite tensor 
category $\cala$ the Eilenberg-Watts equivalences restrict to a commuting triangle
  \be
  \begin{tikzcd}[column sep=-0.6ex,row sep=5.5em]
  ~ & \calz(\caln\opm \boti \calm) \,\simeq\, \calz(\opmmm\caln \boti \calm)
  \ar[xshift=-11pt]{dl}[xshift=-2pt]{\Phile} \ar[xshift=11pt]{dr}[swap]{\Phire} & ~
  \\
  \Funle_\Cala(\caln,\calm) \ar[xshift=-23pt]{ur}[xshift=2pt]{\Psile}
  \ar[yshift=4pt]{rr}{} & ~ 
  & \Funre_\Cala(\caln,\calm) \ar[xshift=23pt]{ul}[swap]{\Psire}
  \ar[yshift=-4pt]{ll}{}
  \end{tikzcd}
  \ee
of two-sided adjoint equivalences of finite categories.
\end{cor}


\section{Eilenberg-Watts calculus for balanced functors} \label{sec:EWbal}

For left $\cala$-modules $\calm$ and $\caln$, besides the category of module functors
there is also another pertinent category of functors with structures related to the
$\cala$-ac\-tions: The category $\Funbal(\caln_{\!\cala} \boxtimes {}_{\cala}\calm^{\#},\vect)$ 
of balanced functors, as described in Definition \ref{def:enoM.3.1}. 
 
As it turns out, the relation between module functors and balanced functors fits well with 
the Eilenberg-Watts calculus. Indeed, notice that by Lemma \ref{Lemma:EW-adjun} it follows
that for any two finite linear categories $\calb$ and $\calc$ there is an equivalence
  \be
  \bearll
  \Xi_{\calb,\calc} : ~~ &
  \Funle(\calb,\calc) \xrightarrow{~\simeq~} \Funle(\calb\boti \overline\calc,\vect) 
  \Nxl2 &
  \qquad \quad F \xmapsto{~\phantom x~} \Xi_{\calb,\calc}(F) \,{:=}\,\Hom_\calc(-,F(-))
  \eear
  \label{eq:catbim}
  \ee
of categories \Cite{Eq.\,(B.5)}{fScS4}.
Let now $\calm$ and $\caln$ be left $\cala$-modules and endow the opposite category 
$\calmopp$ with the right $\cala$-module structure given by $\calm\opm$.
Then for any left exact $\cala$-module functor $F\colon \caln \To \calm$ we get
  \be
  \bearll
  \Xi_{\caln,\calm}(F)(a \Act n \boti\cc m) \!\! & \equiv \Hom_\calm(m,F( a \Act n) )
  \Nxl3 &
  \cong \Hom_\calm(m, a \Act F(n))
  \Nxl3 &
  \cong \Hom_\calm (a^\vee \Act m, F(n))
  \Nxl3 &
  = \Xi_{\caln,\calm}(F)(n \boti \cc m \Ract a ) \,,
  \eear
  \ee
where we first use the module functor structure, then the duality of $\cala$ 
and finally the definition of the action on $\calm\opm$.
It follows that the functor $\Xi_{\calm,\caln}(F)$ is a balanced functor.
Since $\Xi_{\caln,\calm}$ is an equivalence, it follows as well that, conversely,
a balancing on $ \Xi_{\caln,\calm}(F)$ is the same as the structure of a module
functor on $F$. We conclude that $ \Xi_{\caln,\calm}$ induces an equivalence
  \be
  \Funle_{\cala}(\caln,\calm) \simeq\, \Funlebal(\caln \boti \calm\opm, \vect) \,.
  \label{eq:balanceXiF}
  \ee
Together with the results of the previous section we thus arrive at the following 
relations between module structures and the Eilenberg-Watts equivalences: 

\begin{Proposition}\label{Proposition:ModEW-bal}
Let $\caln$ and $\calm$ be left $\cala$-modules. The Eilenberg-Watts equivalences
provide the following sequence of equivalences:
  \be
  \begin{tikzcd}[row sep=3.4em]
  \Funle_{\cala}(\caln,\calm) \ar{d}{\simeq}[swap]{\eqref{eq:balanceXiF}}
  \ar{r}{\eqref{LexZ--RexZ}}[swap]{\simeq}
  & \calz(\caln\opm \boti \calm)
  \ar{d}{Y_{\!\calz(\caln\opm\boxtimes\calm)}}[swap]{\simeq}
  \\
  \Funlebal(\caln \boti \calm\opm, \vect) ~~ & ~~ 
  \Funle(\cc{\calz(\caln\opm \boti \calm)}, \vect)\,.
  \end{tikzcd}
  \label{eq:sequence}
  \ee
\end{Proposition}

The equivalence in the right column of the diagram \eqref{eq:sequence} is the Yoneda 
embedding $Y_{\!\calx}\colon \calx \,{\xrightarrow{\,\simeq}\,} 
            $\linebreak[0]$
\cc{ \cc{\calx}} \boti \vect \,{\xrightarrow{\eqref{Phile.Psile}}}\, \Funle(\cc{\calx}, \vect)$
for $\calx \eq \calz(\caln\opm \boti \calm)$.

It is tempting to expect that the so obtained equivalence
$\Funlebal(\caln \boti \calm\opm, \vect) \, {\simeq}\, \Funle
            $\linebreak[0]$
(\cc{\calz(\caln\opm \boti \calm)}, \vect)$ is part of a family
  \be
  \Funlebal(\caln \boti \calm\opm, \calx) 
  \,\simeq\, \Funle(\cc{\calz(\caln\opm \boti \calm)}, \calx)
  \ee
of equivalences that satisfies the universal property of the (left exact) 
relative De\-lig\-ne product. To prove that this is indeed the case, it is useful 
to express all categories in question as categories of comodules over suitable 
comonads. This will be discussed in the next section.

\begin{Remark}
In \Cite{Def.\,2.1}{shimi20} the related notion of a \emph{module profunctor} between 
left $\cala$-module categories $\caln$ and $\calm$ is introduced: A module profunctor 
$F$ from $\caln$ to $\calm$ is a functor $F\colon \caln \boti \cc{\calm} \To \vect$
together with a collection of coherent dinatural morphisms $F(n \boti \cc{m}) 
                  \,{\to} $\linebreak[0]$
F(a \Act n \boti \cc{a \Act m })$ for all $n\,{\in}\,\caln$,
$m\,{\in}\,\calm$ and $a\,{\in}\,\cala$. In our setting, the category of 
module profunctors is equivalent to the category of module functors. Indeed, such
a collection of dinatural morphisms is the same as a morphism $F \To Z_{[2]}(F)$, with
$Z_{[2]}$ as in \eqref{eq:def:comonadZn} acting on $\Fun(\caln \boti \calm\opm, \vect)$. 
The coherence of the dinatural family is thereby the same as requiring this morphism
$F \To Z_{[2]}(F)$ to be a comodule action. Moreover, this correspondence extends to 
the morphisms between module profunctors. Thus we conclude that the category of 
module profunctors from $\caln$ to $\calm$ is equivalent to the category 
$\Comod_{Z_{[2]}}(\Fun(\caln \boti \calm\opm, \vect))$, which by Lemma \ref{lem:comodules}
is equivalent to the category $\Funbal(\caln \boti \calm\opm, \vect)$. With Proposition
\ref{Proposition:ModEW-bal} we finally conclude that $ \Funle_{\cala}(\caln,\calm) $
is equivalent to the category of left exact module profunctors from $\caln$ to $\calm$.
\end{Remark}


\section{Transporting (co)monads along Eilenberg-Watts}

The situations considered in the previous sections may be thought of as  particular
instances of applying the Eilenberg-Watts equivalence
  \be
  \Funre(\caln \boti \calk, \calx) \,\simeq\, \cc{\caln \boti \calk} \boxtimes \calx 
  \,\simeq\, \Funle(\caln \boti \calk, \calx) 
  \label{eq:EW-cases}
  \ee
to situations in which a finite tensor category $\cala$ has two actions on two of the 
three finite categories $\caln$, $\calk$ and $\calx$: In Section \ref{sec:EWm},
$\cala$ acts from the left on $\caln$ and on $\calx \,{=}\, \calm$, while
$\calk$ is just $\vect$ (which, being a unit for the Deligne product, can be omitted)
and does not carry an $\cala$-action; and in Section \ref{sec:EWbal}, $\cala$ acts on
$\calk \,{=}\, \calm$ and on $\caln$, but does not act on $\calx \,{=}\, \vect$.
For treating these different cases uniformly, the language of monads and comonads
turns out to be convenient; in particular it allows for a clear description of
the relative Deligne product. Indeed, by invoking Lemma \ref{lem:comodules}, the
twisted center, module functors, and balanced functors can all be seen as comodules over 
suitable comonads that are transported via Eilenberg-Watts functors. 

We are now going to show that, when suitable finiteness conditions are satisfied, 
the Eilenberg-Watts calculus yields a convenient description of the categories of
modules over a monad on a functor category. 
As a preparation, we note the following transfer principle:

\begin{Proposition} \label{proposition:induced-monad}
Let $\Phi\colon\,\calx \,{\leftrightarrows}\, \caly \,\colon \Psi$ be an adjoint equivalence
between categories $\calx$ and $\caly$.
\\[2pt]
{\rm (i)}
Given a (co)monad $\TX$ on $\calx$, the endofunctor
  \be
  \Phi \circ \TX \circ \Psi =: \TY
  \ee
has a canonical structure of a (co)monad on $\caly$.
 \\
We refer to the (co)monad $\TY$ as the one \emph{transferred} along the adjoint equivalence. 
\\[2pt]
{\rm (ii)}
The equivalence between $\calx$ and $\caly$ can be lifted in a unique way to an adjoint equivalence 
  \be
  \begin{tikzcd}[column sep=2.1em]
  \widehat\Phi\,\colon\, \text{(Co)mod}_{\TX}(\calx) \ar[yshift=3pt]{r}
  & \text{(Co)mod}_{\TY}(\caly) ~\colon \!\widehat\Psi \ar[yshift=-3pt]{l}
  \end{tikzcd}
  \label{eq:transfer}
  \ee
between the categories of (co)modules over the transferred (co)monads such that the diagrams
  \be
  \begin{tikzcd}[row sep=3.1em]
  \calx \ar{d}[swap,xshift=1pt]{\mathrm{Ind}_{\TX}} \ar{r}{\Phi} & \caly \ar{d}{\mathrm{Ind}_{\TY}}
  \\
  \text{Mod}_{\TX}(\calx) \ar{r}{\widehat\Phi} & \text{Mod}_{\TY}(\caly)
  \end{tikzcd}
  \label{eq:diag-transfer-1}
  \ee
and
  \be
  \begin{tikzcd}[row sep=3.1em]
  \calx \ar{d}[swap,xshift=1pt]{\mathrm{coInd}_{\TX}} \ar{r}{\Phi}
  & \caly \ar{d}{\mathrm{coInd}_{\TY}}
  \\
  \text{Comod}_{\TX}(\calx) \ar{r}{\widehat\Phi} & \text{Comod}_{\TY}(\caly)
  \end{tikzcd}
  \label{eq:diag-transfer-2}
  \ee
commute.
Here for a monad $\TZ$ on $\calx$, the induction functor $\mathrm{Ind}_{T} \colon 
\calx \To \text{Mod}_T(\calx)$ has $\TZ$ as underlying functor, and an analogous statement
holds for the coinduction $\mathrm{coInd}_{T} \colon \calx \To \text{Comod}_T(\calx)$. 
\end{Proposition}

\begin{proof}
The statement is shown by direct computation. Note that for the unit axiom of the induced 
(co)monad $\TY$ on $\caly$ it is necessary that the equivalence is an adjoint equivalence.
The uniqueness of the functor $\widehat{\Phi}$ follows by considering the adjoints of the 
diagrams \eqref{eq:diag-transfer-1} and
\eqref{eq:diag-transfer-2}, which are commuting diagrams involving the forgetful functor. 
\end{proof}

In the sequel we denote the (co)induction functor for a (co)monad $\TZ$ just by its 
underlying functor $\TZ$. 
Given a (co)monad $\TZ \colon \calm \,{\to}\, \calm$, for any category $\calx$ the functor 
category $\Fun(\calm,\calx)$ inherits a (co)monad $\TZ^*_{}$ by pre-composition with $\TZ$.
We call a (co)monad \emph{left exact} iff its underlying endofunctor is left exact. 

\begin{Proposition} \label{Proposition:comonadEW}
Let $\TZ\colon \calm \,{\to}\, \calm$ be a left exact comonad on a finite category $\calm$, let
$\calx$ be a finite category, and let $\TZ^*_{}$ be the induced comonad on $\Funle(\calm,\calx)$. 
\\[3pt]
{\rm (i)} The comonad on the category $\cc{\calm} \boti \calx$ that is obtained by
transferring the co\-mo\-nad $\TZ^*_{}$ on the functor category $\Funle(\calm,\calx)$ 
along the Eilenberg-Watts
equivalence $\Psile \colon \Funle(\calm,\calx) \,{\to}\, \cc{\calm} \boti \calx$ is 
  \be
  \cc{\TZ^{\la}} \boti \Id: \quad
  \cc{\calm} \boti \calx \xrightarrow{~~} \cc{\calm} \boti \calx \,.
  \label{eq:transfer-EW}
  \ee
Moreover, there is a canonical equivalence 
  \be
  \Comod_{\cc{\TZ^{\la}}}(\cc{\calm}) \,\simeq\, \cc{\Comod_{\TZ}(\calm)}
  \label{eq:equ-cat-comod}
  \ee
between the category of $\cc{\TZ^{\la}}$-comodules and the opposite of the
category of $\TZ$-comodules.
\\[3pt]
{\rm (ii)}
Using implicitly the canonical equivalence \eqref{eq:equ-cat-comod}, the Eilenberg-Watts
equivalences \eqref{Phile.Psile} induce an equivalence $\Theta^{\rm l}$ between the category
of $\TZ^{*}_{}$-comodules and the category $\Funle(\Comod_{\TZ}(\calm),
            $\linebreak[0]$
\calx)$ of left exact functors, giving rise to a commuting triangle 
  \be
  \begin{tikzcd}[row sep=3.1em,column sep=-0.9em]
  \Comod_{\TZ^{*}}(\Funle(\calm,\calx )) \arrow[rr,"\Theta^{\rm l}","\simeq"']
  \ar{dr}[left,yshift=-6pt]{\wPsile}
  & ~ & \Funle(\Comod_{\TZ}(\calm),\calx)
  \\
  ~ & \CC{\Comod_{\TZ}(\calm)} \boti \calx \ar{ur}[right,yshift=-3pt]{\Phile} & ~
  \end{tikzcd}
  \label{equation:factor-EW}
  \ee
A quasi-inverse of the equivalence $\Theta^{\rm l}$ is given by 
  \be
  \Funle(\Comod_{\TZ}(\calm),\calx) \ni~ G \,\xmapsto{~~}\,
  G \circ \TZ ~\in \Comod_{\TZ^{*}}(\Funle(\calm,\calx )) \,.
  \ee
 
\smallskip
\noindent
{\rm (iii)}
Concretely, for any $F \,{\in}\, \Comod_{\TZ^{*}}(\Funle(\calm,\calx))$
we have a canonical isomorphism
  \be
  \int^{z \in \Comod_{\TZ}(\calm)}\!\! \cc{z} \;{\boxtimes}\, \Theta^{\rm l}(F)(z)
  \,\cong \int^{m \in \calm}\!\! \cc{m} \;{\boxtimes}\, F(m)
  \label{eq:Main-ModuleEW}
  \ee
of objects in $\CC{\Comod_{\TZ}(\calm)} \boti \calx$.
\end{Proposition}

\begin{proof}
(i)\,
Consider the transfer of the comonad $\TZ^{*}$ on $\Funle(\calm,\calx)$ 
to $\CC{\calm} \boti \calx $ along the Eilenberg-Watts equivalence $\Psile$. It follows 
from \eqref{eq:lem:fScS2:3.8} that this transferred comonad is $\CC{\TZ^{\la}}\boti \Id$. 
By considering, for a comodule $m$ with coaction $m \To Tm$ in $\calm$, the morphism in
$\calmopp$ that corresponds to $T^{\la}m \To m$, one then sees that there is an equivalence 
$ \Comod_{\CC{\TZ^{\la}}}(\CC{\calm}) \,{\simeq}\, \CC{\Comod_{\TZ}(\calm)}$.
\\[3pt]
(ii)\, 
The existence of an induced equivalence $\Theta^{\rm l}$ as in \eqref{equation:factor-EW} 
readily follows from Proposition \ref{proposition:induced-monad} with the help of the 
Eilenberg-Watts equivalences. To show the statement about the quasi-inverse, we first
apply the quasi-inverse $\Psile$ of $\Phile$ in \eqref{equation:factor-EW} to 
$G \,{\in}\, \Funle(\Comod_{\TZ}(\calm),\calx)$ and then use the quasi-inverse $\wPhile$
to $\wPsile$ to arrive at
  \be
  \bearll
  \wPhile(\Psile(G)) \!\!\! & \displaystyle
  = \int^{z \in \Comod_{T}(\calm)}\!\! \Hom_{\calm}(U(z),-) \otimes G(z) 
  \Nxl1 & \displaystyle
  \cong \int^{m \in \calm}\! \Hom_{\calm}(m,-) \otimes G(\TZ (m)) 
  \,\cong\, G \circ \TZ \,.
  \eear
  \label{eq:pseudo=inv}
  \ee
Here the second step uses that $\TZ^{\la} \,{=}\, U\colon \Comod_{T}(\calm) \,{\to}\, \calm$
is the forgetful functor together with Lemma \ref{lem:3.8plus}.
\\[2pt]
Part (iii) is a direct consequence of part (ii).
\end{proof}

\begin{Remark}
(i) Note that as compared with the ordinary Eilenberg-Watts situation in \eqref{Phile.Psile}, 
in the diagram \eqref{equation:factor-EW} we deal with three instead of two categories --
the adjoint pair of quasi-inverse functors $\Phi$ and $\Psi$ splits up into a triangle of 
equivalences. (Also, in \eqref{equation:factor-EW} we use the same
notation $\widehat\Psi$ as in \eqref{eq:transfer} even though
we actually deal with the composition of the functor $\widehat\Psi$ from
\eqref{eq:transfer} with the canonical equivalence \eqref{eq:equ-cat-comod}.)
\\[2pt]
(ii) It is worth to note that the functor $\Theta^{\rm l}(F)$ in \eqref{eq:Main-ModuleEW} 
can be defined point-wise through the equalizer 
  \be
  \Theta^{\rm l}(F)(z)
  \xrightarrow{~~~~} F(U(z)) \xrightrightarrows[\varphi_2]{~~\varphi_1~} F(T(z)) \,,
  \label{eq:EquializerF}
  \ee
where the morphisms $\varphi_1$ and $\varphi_2$ are defined via the $T_*$-comodule 
structure of $F$ and by the $T$-comodule structure of $z$, respectively. 
\end{Remark}

Dually, for a right exact monad $T$ on $\calm$ we obtain an equivalence involving categories
of modules over monads:
  \be
  \Theta^{\rm r} :\quad
  \Mod_{T^{*}_{}}(\Funre(\calm,\calx)) \,\simeq\, \Funre(\Mod_{T}(\calm),\calx) \,.
  \label{eq:re-version}
  \ee

We are now in a position to show that the Eilenberg-Watts calculus can be lifted to 
twisted centers. To this end we apply Proposition \ref{Proposition:comonadEW} to the case that
$T$ is the central comonad $Z_{[\kappa]}$ for a $\cala$-bimodule $\calm$.
For the category  $\Funle^{\kappa}(\calm,\calx)$ of $\kappa$-balanced functors, we then get
  
\begin{thm} \label{thm:EW:coendzFz}
Let $\calm$ be a finite bimodule category over a finite tensor category $\cala$,
$\calx$ a finite category, and $\kappa \,{\in}\, 2\Z$.
\\[2pt]		
{\rm (i)} The Eilenberg-Watts calculus provides explicit equivalences
  \be
  \begin{tikzcd}[row sep=2.8em]
  \Funle^{\kappa}(\calm,\calx ) \ar{rr}{\Theta^{\rm l}} \ar{dr}[left,yshift=-4pt]{\wPsile}
  &~& \Funle(\cent^\kappa(\calm),\calx)
  \\
  ~& \CC{\cent^\kappa(\calm)} \boti \calx  \ar{ur}[right,yshift=-3pt]{\Phile}  &~
  \end{tikzcd}
  \label{equation:factor-EW-module}
  \ee
of linear categories for all $\calx$. A quasi-inverse of the equivalence $\Theta^{\rm l}$ 
is given by 
  \be
  \Funle(\cent^\kappa(\calm),\calx) \ni~ G \,\xmapsto{\phantom{~~}}\,
  G \circ Z_{[\kappa]} ~\in  \Funle^{\kappa}(\calm,\calx )\,.
  \ee
In particular, for $\calm \eq \calk_\cala \boti {}_\cala \caln$
with $\calk$ and $\caln$ finite right and left $\cala$-modules, respectively, and 
$\kappa \eq 0$, the category $\cent^0(\calk \boxtimes \caln)$ becomes a left exact relative 
Deligne product  of $\calk$ and $\caln$ according to Definition \ref{definition:relDel}. 
\\[3pt]
{\rm (ii)}
For any left exact $\kappa$-balanced functor $F\colon \calm \,{\to}\, \calx$ there is 
an isomorphism 
  \be
  \int^{z \in \cent^\kappa(\calm)}\!\! \cc{z} \boxtimes \Theta^{\rm l}(F)(z) \,\cong
  \int^{m \in \calm}\!\! \cc{m} \boxtimes F(m) \,
  \label{eq:Main-ModuleEW-concret}
  \ee
of objects in $\CC{\cent^\kappa(\calm)} \boti \calx$. 
\\[3pt]
{\rm (iii)}
Specifically, for the co-induction functor $Z_{[\kappa]}\colon \calm \,{\to}\,
\cent^\kappa_{}(\calm)$ that corresponds to the comonad $Z_{[\kappa]}$, the corresponding
functor $\Theta^{\rm l}(Z_{[\kappa]})\colon \cent^\kappa(\calm) \,{\to}\, \cent^\kappa_{}(\calm)$ 
is the identity functor, whereby Equation \eqref{eq:Main-ModuleEW-concret} reduces to 
  \be
  \bearll \displaystyle
  \int^{z \in \cent^\kappa_{}(\calm)}\!\! \cc{z} \boxtimes \Id(z) \!\!\!& \displaystyle
  \cong \int^{m\in\calm}\!\! \cc{m} \boxtimes Z_{[\kappa]}(m)
  \Nxl2 & \displaystyle
  \cong \int^{m\in\calm}\!\! \cc{m} \,\boxtimes\! \int_{\!a\in\cala}\! a \act m \ract a^{[\kappa-1]} .
  \label{eq:EW:coendzz=...}
  \eear
  \ee
\end{thm}
 
\begin{proof}
In view of Proposition \ref{Proposition:comonadEW}, it only remains to be shown
that the category $\Comod_{(Z_{[\kappa]})^{*}}
            $\linebreak[0]$
(\Funle(\calm,\calx ))$ is equivalent to the category $\Funle^{\kappa}(\calm,\calx)$. By 
Lemma \ref{lem:comodules}, the latter category is equivalent to the category of comodules 
over the comonad $Z_{[2-\kappa]}$ on $\Funle(\calm,\calx )$. Now by direct computation
one verifies that $(Z_{[\kappa]})^{*} \,{=}\, Z_{[2-\kappa]}$ as 
comonads. Thus the statement follows. 
\end{proof}

We refer to the functors $\wPsile$ and $\Phile$ in the commuting triangle
\eqref{equation:factor-EW-module} 
as the \emph{$($left exact$)$ module Eilenberg-Watts equivalences}.

By applying this result to the opposite category, an isomorphism analogous to 
\eqref{eq:EW:coendzz=...} holds for ends instead of coends: We have an isomorphism
  \be
  \int_{\!z\in \cent^\kappa(\calm)}\!\! \cc z\boxtimes z
  \ {\cong} \int_{m\in\calm}\! \cc m\boxtimes Z^{[\kappa]}(m)
  \ee
of objects in $\CC{\cent^\kappa(\calm)} \boti \cent^\kappa(\calm)$,            
where $Z^{[\kappa]}$ is the $\kappa$-twisted central monad, which has
$m \,{\mapsto}\! \int^{a\in\cala}\! a \act m \ract a^{[\kappa-3]}$ as underlying 
endofunctor \Cite{Eq.\,(3.17)}{fScS4}.

\begin{Remark}
Recall that the category of comodules over the central comonad \eqref{eq:defZ}
of a finite tensor category $\cala$ is equivalent to the center $\calz(\cala)$. Thus
in this case the result \eqref{eq:EW:coendzz=...} gives the isomorphism
  \be
  \int^{z \in \calz(\cala)}\!\! \cc{z} \;{\boxtimes}\, z \,\cong
  \int^{b \in \cala}\! \cc{b} \;{\boxtimes} \int_{\!a\in\cala}\! a \act b \ract a^{\vee}
  \label{eq:Coend-Center}
  \ee
of objects in $\cc{\calz(\cala)} \boti \calz(\cala)$, where the two balancings on 
each the two tensor factors on the right hand side are provided by the isomorphisms 
\eqref{eq:Action-boxtimes}. Note that the balancing on the first factor involves
both the $b$- and the $a$-variables.
This reproduces the description of the coend of the center of a fusion category,
as given in \Cite{Sect.\,9.3}{brVi5} and recalled in \Cite{Sect.\,9.2.3}{TUvi}. Indeed,
in that work (note their different convention for duals in $\cala$) it is shown that 
  \be
  \int^{z\in\calz(\cala)}\!\!
  z^{\vee} {\otimes}\, z \,~\cong \int^{b\in\cala}\!\! T(b)^{\vee} \oti b
  \label{eq:Coend-Center-BV}
  \ee
as objects in $\calz(\cala)$, where $T(b)$ appearing on the right hand side is the coend
$T(b) \,{=}\, \int^{a} a^\vee {\otimes}\, b \oti a$, so that
$T(b)^\vee \,{\cong}\, \int_a  a^{\vee} {\otimes}\, {b}^{\vee} {\otimes}\, {a}^{\vee\vee}$.
If we combine \eqref{eq:Coend-Center} with the isomorphism \eqref{eq:Action-boxtimes} and 
then apply the exact functor from $\cc{\cala} \boti \cala$ to $\cala$ that is given by 
$\cc{a} \boti b \,{\mapsto}\, {a}^{\vee} {\otimes}\, b$, we obtain  
  \be
  \bearll \displaystyle
  \int^{z \in \calz(\cala)}\!\!\! {z}^{\vee} {\otimes}\, z \!\!\!&\displaystyle
  \cong \int^{b\in\cala}\!\!\!\!
  \int_{\!a\in\cala} (a^\vee {\otimes}\, b \oti a )^\vee \oti b
  \Nxl2 & \displaystyle
  \cong \int^{b\in\cala}\!\!\!\! \int_{\!a\in\cala}\!
  {a}^{\vee} {\otimes}\, {b}^{\vee} {\otimes}\, a^{\vee \vee} {\otimes}\, b \,,
  \eear
  \ee
thus reproducing \eqref{eq:Coend-Center-BV}. 
 \end{Remark}

\begin{Remark}
In the semisimple case, i.e.\ if $\cala$ is a fusion category, ends and coends coincide,
and \eqref{eq:Coend-Center} reduces to an isomorphism
  \be
  \bigoplus_{\alpha} \cc{z_\alpha} \;{\boxtimes}\, z_\alpha \,\cong~
  \bigoplus_{i,j} \cc{b_i} \;{\boxtimes}\, a_j \act b_i \ract a_j^{\vee}
  \label{eq:Coend-Center-ses}
  \ee
of objects in $\calz(\cala)$,
where $\alpha$ ranges over a set of representatives for the isomorphism classes of 
simple objects of $\calz(\cala)$, while $i$ and $j$ range over a set representing the 
isomorphism classes of simples in $\cala$. The isomorphism \eqref{eq:Coend-Center-ses} 
relates $\cala$ and its Drinfeld center. It leads to  Lemma 7.4 of \cite{balKi} after 
applying the functor $\cc{\calz(\cala)} \boti \calz(\cala) \To \calz(\cala)$ that is 
given by $\cc{x} \boti y \,{\mapsto}\, x^{\vee} {\otimes}\, y$ to both sides. 
In case $\cala$ has a spherical structure, the isomorphism \eqref{eq:Coend-Center-ses} 
plays a central role in approaches \cite{balKi,tuVi} to relate the Turaev-Viro TFT 
for a spherical fusion category $\cala$ with the Reshetikhin-Turaev 
TFT for $\calz(\cala)$, in which it is responsible for the factorization property of 
the state spaces: It is the algebraic counterpart of factoring a surface that has been
obtained by gluing two surfaces along a common circular boundary, compare
\Cite{Thm.\,5.2}{fScS4}. 
\end{Remark}

\begin{Remark}
Statements analogous to those of Theorem \ref{thm:EW:coendzFz} hold for right exact 
functors. In particular we have explicit equivalences 
  \be
  \Funre^{\kappa-4}(\calm,\calx)\,\simeq\,\Funre(\cent^\kappa(\calm),\calx) \,.
  \ee
By considering the $\cala$-bimodule $\calm \,{=}\, \caln_\Cala\,{\boxtimes}\, {}_{\cala}\caln'$,
where $\caln$ and $\caln'$ are a right and a left $\cala$-mo\-du\-le, respectively,
with Proposition \ref{proposition:rel-Del} we then recover for $\kappa \,{=}\, {+}4$ the 
statement that the right exact relative Deligne product is equivalent to the twisted center.
Recalling from \eqref{eq:NoM-Lex-Rex} that the Eilenberg-Watts equivalences provide 
an equivalence of the categories of left and right exact functors between finite categories,
we arrive at the following module version of that picture:
  \be
  \begin{tikzcd}[row sep=2.6em]
  \Funle^{\kappa}(\calm,\calx ) \ar{dd} \ar{rr}{\Theta^{\rm l}} \ar{dr}[left,yshift=-5pt]{\wPsile}
  & ~ & \Funle(\cent^\kappa(\calm),\calx) \ar{dd}
  \\
  ~ & \CC{\cent^\kappa(\calm)} \boti \calx \ar{ur}[right,yshift=-4pt]{\Phile} \ar{dr}{\Phire} & ~
  \\
  \Funre^{\kappa-4}(\calm,\calx ) \ar{rr}{\Theta^{\rm r}} \ar{ur}[yshift=-2pt]{\wPsire}
  & ~ & \Funre(\cent^\kappa(\calm),\calx)_{\phantom|}
  \end{tikzcd}
  \label{equation:factor-EW-module-le-re}
  \ee
where the unlabeled vertical arrows are defined by the commutativity of the diagram. 
\end{Remark}

Combining the module Eilenberg-Watts equivalences \eqref{equation:factor-EW-module}
with Lemma \ref{Lemma:EW-adjun} we obtain the following generalization of Corollary
\ref{cor:thecor}, which is a module version of Lemma \ref{Lemma:EW-adjun}:

\begin{Lemma} \label{Lemma:Module-ew-adj}
Let $\calm$ and $\caln$ be finite left modules over a finite tensor category $\cala$.
For any finite category $\calx$ the equivalences from Lemma \ref{Lemma:EW-adjun} 
induce equivalences
  \be
  \Funle_{\cala}(\caln, \calm \boti \calx)
  \,\simeq\, \Funlebal(\caln \boti \calm \opm, \calx) 
  \,\simeq\, \Funle(\calz^{0}(\caln \boti \calm \opm),\calx) \,.
  \label{eq:module-=EW-adj}
  \ee
\end{Lemma}

\begin{proof}
The first equivalence is a straightforward generalization of Proposition 
\ref{Proposition:ModEW-bal}. The second is Theorem \ref{thm:EW:coendzFz}(i).
\end{proof}

The case of bimodule functors is treated as follows. For bimodules 
${}_{\cala}\calm_{\calb}$ and ${}_{\cala}\caln_{\calb}$, there are two corresponding 
comonads on $\Funle(\caln,\calm)$, with obvious distributive law, whose comodules 
are the corresponding categories of bimodule functors. For the $\cala$-action this 
is $Z_{[2],\cala}$. see Lemma \ref{lem:comodules}(ii). For the $\calb$-action, 
the relevant comonad $Z_{[2],\calb}$ on $F\colon \caln \To \calm$ is given by
$Z_{[2],\calb}(F)(n) \,{=}\, \int_{b \in \calb}F(n \Ract b) \Ract b^{\vee}$. 
To summarize, $Z \,{:=}\, Z_{[2],\cala} \,{\circ}\, Z_{[2],\calb}$ is
canonically a comonad on the category $\Funle(\caln,\calm)$ whose category of 
comodules is equivalent to the category $\Funle_{\cala,\calb}(\caln,\calm)$
of bimodule functors.

\begin{cor} \label{cor:lex-framed-app}
Let ${}_{\cala}\calm_{\calb}$ and ${}_{\cala}\caln_{\calb}$ be finite bimodules 
over finite tensor categories $\cala$ and $\calb$.
The Eilenberg-Watts equivalences induce an equivalence 
  \be
  \Funle_{\cala,\calb}(\caln,\calm)
  \,\simeq\, \calz_{\cala,\calb}(\caln\opm \boti \calm)
  \label{eq:equ-lex-framed-cent-0}
  \ee
of categories, where the category $\caln\opm$ is equipped with the left $\calb$-action 
$b \Act \cc{m} \,{=}\, \cc{m\Ract {}_{}^{\vee\!}b}$ and $ \calz_{\cala,\calb}$ 
denotes the center with respect to the $\cala$ and $\calb$-actions. 
\end{cor}

\begin{proof}
According to Proposition \ref{Proposition:comonadEW} we need to transfer the comonad $Z$ along
the Eilenberg-Watts equivalence $\Funle(\caln,\calm) \,{\simeq}\, \cc{\caln} \boti\calm$. A
straightforward computation yields the comonad on $\cc{\caln} \boti\calm$ that is given by 
  \be
  \cc{n} \boti m \,\longmapsto \int_{\!a\in\cala} \int_{\!b\in\calb}
  \cc{a^{\vee\vee} \act n \Ract {}^{\vee \vee\!}b} \,\boxtimes a \act n \Ract b \,.
  \ee
The category of comodules over this comonad is $\calz_{\cala,\calb}(\caln\opm \boti \calm)$.
\end{proof}

\begin{Remark}
Given any three $\cala$-$\calb$-bimodules ${}_{\cala}\calm_{\calb}$, ${}_{\cala}\caln_{\calb}$ 
and ${}_{\cala}\calk_{\calb}$, the equivalence \eqref{eq:equ-lex-framed-cent-0} induces
a composition operation on the centers: The composition
  \be
  \Funle_{\cala,\calb}(\caln,\calm) \times \Funle_{\cala,\calb}(\calm,\calk) 
  \,\xrightarrow{\phantom{x}}\, \Funle_{\cala,\calb}(\caln,\calk)
  \ee
of bimodules functors yields a functor 
  \be
  \calz_{\cala,\calb}(\caln\opm \boti\calm) \boxtimes \calz_{\cala,\calb}(\calm\opm \boti \calk) 
  \,\xrightarrow{\phantom{x}}\, \calz_{\cala,\calb}(\caln\opm \boti\calk) \,.
  \label{eq:induced-comp}
  \ee
This is given by contraction with the $\Hom$ functor:
  \be
  (\cc{n} \boti m) \boxtimes (\cc{m'} \boti k)
  \xmapsto{\phantom{x}} \Hom_\calm(m',m) \otimes (\cc{n} \boti k) \,.
  \ee
This is seen as follows. In \Cite{Cor.\,3.7}{fScS2} the corresponding functor in the
absence of bimodule structures is shown to be given by the $\Hom$-functor. 
Since the composition of functors induces a composition of bimodule functors, the induced 
composition \eqref{eq:induced-comp} is given by the $\Hom$-functor as well. 
The balancing on the right hand side follows by combining the balancings on the left
and the dualities of $\cala$ and $\calb$ inside the $\Hom$ space.
\end{Remark}

\vskip 3em

\noindent
{\sc Acknowledgements:}\\[.3em]
JF is supported by VR under project no.\ 2017-03836. CS is partially supported by the
RTG 1670 ``Mathematics inspired by String theory and Quantum Field Theory''
and by the Deutsche Forschungsgemeinschaft (DFG, German Research Foundation) under
Germany's Excellence Strategy - EXC 2121 ``Quantum Universe''- QT.2.

     \newpage

\newcommand\wb{\,\linebreak[0]} \def\wB {$\,$\wb}
\newcommand\Bi[2]    {\bibitem[#2]{#1}}
\newcommand\inBo[8]  {{\em #8}, in:\ {\em #1}, {#2}\ ({#3}, {#4} {#5}), p.\ {#6--#7} }
\newcommand\J[7]     {{\em #7}, {#1} {#2} ({#3}) {#4--#5} {{\tt [#6]}}}
\newcommand\JO[6]    {{\em #6}, {#1} {#2} ({#3}) {#4--#5} }
\newcommand\BOOK[4]  {{\em #1\/} ({#2}, {#3} {#4})}
\newcommand\Prep[2]  {{\em #2}, preprint {\tt #1}}
\def\adma  {Adv.\wb Math.}
\def\alrt  {Algebr.\wb Represent.\wB Theory}         
\def\apcs  {Applied\wB Cate\-go\-rical\wB Struc\-tures}
\def\coma  {Con\-temp.\wb Math.}
\def\imrn  {Int.\wb Math.\wb Res.\wb Notices}
\def\jims  {J.\wb Indian\wb Math.\wb Soc.}
\def\joal  {J.\wB Al\-ge\-bra}
\def\joms  {J.\wb Math.\wb Sci.}
\def\pams  {Proc.\wb Amer.\wb Math.\wb Soc.}
\def\quto  {Quantum Topology}
\def\tams  {Trans.\wb Amer.\wb Math.\wb Soc.}

\end{document}